\documentclass[11pt,a4paper]{amsart}
\usepackage[T1]{fontenc}
\usepackage[utf8]{inputenc}
\usepackage{lmodern}
\usepackage{amsmath,amssymb,amsthm}
\usepackage{microtype}
\usepackage{needspace}
\usepackage[a4paper,textwidth=158mm,textheight=235mm,centering]{geometry}
\usepackage[hidelinks]{hyperref}
\hypersetup{
  pdftitle={A sharp bound for additive complements},
  pdfsubject={Additive complements and the sharp constant three halves},
  pdfkeywords={additive complements, counting functions, representation functions, convolution},
  pdfauthor={}
}
\numberwithin{equation}{section}
\newtheorem{theorem}{Theorem}[section]
\newtheorem{lemma}[theorem]{Lemma}
\newcommand{\Nzero}{\mathbb{N}_0}
\newcommand{\ind}{\mathbf{1}}
\newcommand{\abs}[1]{\lvert #1\rvert}

\allowdisplaybreaks[1]

\title[A sharp bound for additive complements]{A sharp bound for additive complements}
\author{Zihan Zhang}
\address{School of Mathematical Sciences and LPMC,
Nankai University, Tianjin 300071, People's Republic of China}
\email{2211056@mail.nankai.edu.cn}
\date{}

\begin{document}

\begin{abstract}
Let $A$ and $B$ be infinite sets of nonnegative integers whose sum contains all sufficiently large integers, and let $A(x)$ and $B(x)$ be their counting functions. We give a discrete proof of the sharp implication
\[
  \limsup_{x\to\infty}\frac{A(x)B(x)}{x}<\frac32
  \quad\Longrightarrow\quad
  A(x)B(x)-x\longrightarrow+\infty.
\]
The constant $3/2$ is best possible. Using finite convolution sums, we show that a long gap in one set forces a gap in the other. We then compare the counting functions before and after the gap to obtain the sharp bound.
\end{abstract}

\maketitle

\section{Introduction}
\label{sec:introduction}

We write $\Nzero=\{0,1,2,\ldots\}$. Two infinite sets $A,B\subseteq\Nzero$ are called additive complements if $A+B$ contains all sufficiently large integers. For real $x\ge0$, put
\[
  A(x)=\#(A\cap[0,x]),\qquad B(x)=\#(B\cap[0,x]).
\]
Every sufficiently large integer at most $x$ has a representation using elements at most $x$, so $A(x)B(x)\ge x-O(1)$.

Hanani asked whether every pair of infinite additive complements satisfies
\[
  \limsup_{x\to\infty}\frac{A(x)B(x)}{x}>1.
\]
Erd\H{o}s recorded this question in \cite[Problem~12, p.~295]{Erdos57}. Danzer \cite{Danzer64} answered it negatively by constructing complements for which
\begin{equation}
  \frac{A(x)B(x)}{x}\longrightarrow1.
  \label{eq:exact}
\end{equation}
Following Ruzsa \cite[Section~1]{Ruzsa17}, such pairs are called exact additive complements. This terminology refers to the counting functions and does not require unique representations.

Danzer also asked whether \eqref{eq:exact} forces
\begin{equation}
  A(x)B(x)-x\longrightarrow+\infty.
  \label{eq:excess}
\end{equation}
S\'ark\"ozy and Szemer\'edi \cite{SS94} proved that it does. Ruzsa \cite[Theorem~1.3]{Ruzsa17} later showed that the excess can grow arbitrarily slowly along a sequence: for any function tending to infinity, there are exact complements whose excess is smaller than that function at arbitrarily large arguments.

Fang and Chen \cite{FC10} proved \eqref{eq:excess} under the weaker assumption
\[
  \limsup_{x\to\infty}\frac{A(x)B(x)}{x}<\frac54.
\]
They replaced $5/4$ by $3-\sqrt3$ in \cite[Theorem~1.1]{FC14}, and Yang and Zhao \cite[Theorem~1.1]{YZ26} obtained $4/3$. Chen and Fang \cite{CF11} also constructed complements with upper limit $3/2$ and excess equal to $1$ at infinitely many integers. Thus $3/2$ is the largest possible threshold.

A measure-theoretic proof of the sharp $3/2$ bound was announced by van Doorn, Liu and Tang on 5~August~2026 in the discussion of Erd\H{o}s Problem~785 \cite{VLT26}. The announcement credits GPT-5.6 Sol with obtaining the proof. The accompanying note takes weak limits of normalized counting measures and uses Gabardo and Lai's characterization \cite{GL14} of probability measures whose convolution is Lebesgue measure on $[0,1]$. However, they think the proof by is too surprising that they are working now to find a human-generated version. We give a proof of the same conclusion using elementary counting and finite convolutions.

\begin{theorem}
\label{thm:main}
Let $A$ and $B$ be infinite additive complements in $\Nzero$. If
\begin{equation}
  \limsup_{x\to\infty}\frac{A(x)B(x)}{x}<\frac32,
  \label{eq:main-hypothesis}
\end{equation}
then
\[
  A(x)B(x)-x\longrightarrow+\infty.
\]
\end{theorem}

The proof begins with a standard observation: if the excess stays bounded along a sequence, then every sufficiently large integer has a unique representation as a sum. This reduction was already used in \cite[p.~238]{SS94} and \cite[p.~85]{FC14}. The upper limit assumption then gives consecutive elements $x<y$ of one set with $y\ge3x+n_0$, where $n_0$ is fixed.

Our use of these gaps is motivated by the $3/2$ example of Chen and Fang \cite{CF11}, discussed in Section~\ref{sec:sharpness}. We prove, uniformly for integers $0\le t\le x$, that
\[
  A(y+t)=A(x)+A(t)+O(1),\qquad
  A(x)B(y+t)=y+O(1).
\]
The first estimate describes how $A$ grows after the gap; the second controls $B$ on the same interval. At $t=x$, they give
\[
  \frac{A(y+x)B(y+x)}{y+x}
  \ge (2-o(1))\frac{y}{y+x}\ge\frac32-o(1)
\]
along the sequence of gaps. The reflected convolution identity used below also appears in \cite[Lemma~3.1]{YZ26}. The longer gap lets us show that certain coefficients vanish exactly, which yields both counting estimates.

Section~\ref{sec:preliminaries} collects the basic lemmas. Sections~\ref{sec:reduction} and~\ref{sec:relations} give the gaps and the two counting estimates. We finish the proof in Section~\ref{sec:main-proof} and verify sharpness in Section~\ref{sec:sharpness}.

\section{Preliminary lemmas}
\label{sec:preliminaries}

For $n\in\Nzero$, let
\[
  r(n)=\#\{(a,b)\in A\times B\mid a+b=n\}.
\]
For a nonnegative integer $X$, also let
\[
  T(X)=\#\{(a,b)\in A\times B\mid a\le X,\ b\le X,\ a+b>X\}.
\]
Splitting the pairs according to whether their sum exceeds $X$ gives
\begin{equation}
  A(X)B(X)=\sum_{n=0}^{X}r(n)+T(X).
  \label{eq:rectangle}
\end{equation}

We include the standard counting argument from \cite[p.~238]{SS94}; see also \cite[p.~85]{FC14}.

\begin{lemma}
\label{lem:bounded}
Suppose that $A$ and $B$ are infinite additive complements and that $A(x)B(x)-x$ does not tend to $+\infty$. There is an integer $n_0\ge1$ such that
\begin{equation}
  r(n)=1\qquad(n\ge n_0).
  \label{eq:eventual}
\end{equation}
There is also a strictly increasing sequence of integers $X_j$ tending to infinity such that
\begin{equation}
  A(X_j)B(X_j)=X_j+O(1),\qquad T(X_j)=O(1).
  \label{eq:tight-sequence}
\end{equation}
\end{lemma}

\begin{proof}
Choose $N\ge1$ such that $r(n)\ge1$ for $n\ge N$. The excess is bounded below. Since it does not tend to infinity, it is bounded above along some sequence tending to infinity. Taking integer parts and then a subsequence gives strictly increasing integers $X_j$ with $A(X_j)B(X_j)=X_j+O(1)$.

For every integer $X\ge N$, \eqref{eq:rectangle} gives
\[
  A(X)B(X)-X
  \ge 1-N+\sum_{n=N}^{X}\bigl(r(n)-1\bigr).
\]
The sum has nonnegative integer terms and is bounded along $X_j$. Its partial sums are nondecreasing, so only finitely many terms are nonzero. This proves \eqref{eq:eventual} for some $n_0$. It follows that $\sum_{n=0}^{X}r(n)=X+O(1)$, and \eqref{eq:rectangle} then gives $T(X_j)=O(1)$.
\end{proof}

We also need a bound on the number of representations of a difference. For $n\ge0$, define
\[
\begin{aligned}
  d_{A,B}(n)&=\#\{(a,b)\in A\times B\mid a-b=n\},\\
  d_{B,A}(n)&=\#\{(a,b)\in A\times B\mid b-a=n\}.
\end{aligned}
\]

\begin{lemma}
\label{lem:differences}
Suppose that $r(n)=1$ for every $n\ge n_0$. If two distinct pairs satisfy
\[
  b-a=b'-a',\qquad (a,b),(a',b')\in A\times B,
\]
then $a,a',b,b'<n_0$. In particular,
\begin{equation}
  d_{A,B}(n)\le1,\qquad d_{B,A}(n)\le1
  \qquad(n\ge n_0).
  \label{eq:eventual-differences}
\end{equation}
\end{lemma}

\begin{proof}
The equality of differences gives $a+b'=a'+b$. Distinct pairs have $a\ne a'$ and $b\ne b'$, so these are two distinct representations of a sum. By eventual uniqueness, their common value is less than $n_0$. Nonnegativity gives $a,a',b,b'<n_0$, and hence $|b-a|<n_0$. This proves both inequalities in \eqref{eq:eventual-differences}.
\end{proof}

To keep track of the finitely many exceptions in \eqref{eq:eventual}, put
\begin{equation}
  \delta(n)=r(n)-1\quad(n\ge0),\qquad
  E=\sum_{n=0}^{n_0-1}|\delta(n)|,\qquad C=E+n_0.
  \label{eq:fixed-errors}
\end{equation}
Thus $\delta$ is supported on $[0,n_0-1]$, and for every integer $N\ge n_0-1$,
\begin{equation}
  \sum_{n=0}^{N}r(n)=N+1+c,
  \qquad c=\sum_{n=0}^{n_0-1}\delta(n),\qquad |c|\le E.
  \label{eq:sum-count}
\end{equation}
The constants $E$, $C$ and $c$ depend only on the values of $r(n)$ below $n_0$, so they do not depend on the gaps chosen later.

For later use, write $\ind_S$ for the indicator of a set $S\subseteq\Nzero$. All functions are extended by zero to negative integers. Their convolution is
\[
  (u*v)(n)=\sum_{j=0}^{n}u(j)v(n-j)\qquad(n\ge0).
\]
Convolution is commutative and associative, and convolution with $h=\ind_{\Nzero}$ gives partial sums.

\section{Long gaps from bounded excess}
\label{sec:reduction}

We first find an empty upper half of an initial segment, as in \cite[Section~4]{YZ26}. A second comparison, at $X_j-x_j$, gives a gap extending beyond $3x_j$.

\Needspace{10\baselineskip}
\begin{lemma}
\label{lem:gaps}
Suppose that the assumptions of Lemma~\ref{lem:bounded} hold. Assume in addition that there is a constant $\kappa$ with $1<\kappa<3/2$ such that
\begin{equation}
  A(t)B(t)\le\kappa t
  \qquad\text{for all sufficiently large real }t.
  \label{eq:global-bound}
\end{equation}
Fix $n_0$ as in \eqref{eq:eventual}. After interchanging $A$ and $B$ if necessary, there are consecutive elements $x_j<y_j$ of $A$ such that $x_j\to\infty$ and
\begin{equation}
  x_j\ge2n_0,\qquad y_j\ge3x_j+n_0.
  \label{eq:large-gaps}
\end{equation}
\end{lemma}

\begin{proof}
Take $X_j$ from Lemma~\ref{lem:bounded}. Since $A(X_j)B(X_j)=X_j+O(1)$ and both counting functions tend to infinity,
\begin{equation}
  A(X_j)=o(X_j),\qquad B(X_j)=o(X_j).
  \label{eq:small-counts}
\end{equation}
Let
\[
  u_j=A(X_j)-A(X_j/2),\qquad
  v_j=B(X_j)-B(X_j/2).
\]
Pairs with both elements greater than $X_j/2$ are counted by $T(X_j)$, so
\begin{equation}
  u_jv_j\le T(X_j)=O(1).
  \label{eq:upper-halves}
\end{equation}
If both $u_j$ and $v_j$ were positive for infinitely many $j$, they would both be bounded on that subsequence. Hence
\begin{align*}
  A(X_j/2)B(X_j/2)
  &=(A(X_j)-u_j)(B(X_j)-v_j)\\
  &=X_j+O(A(X_j)+B(X_j)+1)=X_j+o(X_j).
\end{align*}
This contradicts \eqref{eq:global-bound} at $X_j/2$, since $\kappa<2$.

Thus at least one of $u_j,v_j$ is zero for all large $j$. After interchanging the sets and passing to a subsequence, we may assume that $A\cap(X_j/2,X_j]=\varnothing$. Put $x_j=\max(A\cap[0,X_j])$. Then $x_j\to\infty$ and
\begin{equation}
  x_j\le X_j/2,\qquad A\cap(x_j,X_j]=\varnothing.
  \label{eq:initial-gap}
\end{equation}

We now compare the counting functions at $X_j-x_j$. Every $b\in B\cap(X_j-x_j,X_j]$ gives a pair $(x_j,b)$ counted by $T(X_j)$, so
\[
  0\le B(X_j)-B(X_j-x_j)\le T(X_j)=O(1).
\]
Also $A(X_j-x_j)=A(X_j)$ by \eqref{eq:initial-gap}. Therefore
\begin{equation}
  A(X_j-x_j)B(X_j-x_j)
  =X_j+O(A(X_j)+1)=X_j+o(X_j).
  \label{eq:shifted-count}
\end{equation}
Applying \eqref{eq:global-bound} at $X_j-x_j$ gives $X_j+o(X_j)\le\kappa(X_j-x_j)$, or
\begin{equation}
  \frac{x_j}{X_j}\le\frac{\kappa-1}{\kappa}+o(1).
  \label{eq:relative-gap}
\end{equation}
Since $(\kappa-1)/\kappa<1/3$, there is a fixed $\varepsilon>0$ such that $X_j-3x_j\ge\varepsilon X_j$ for all large $j$.

Let $y_j$ be the next element of $A$ after $x_j$. By \eqref{eq:initial-gap}, $y_j>X_j$, so $y_j\ge3x_j+n_0$ for large $j$. Since $x_j\to\infty$, we also have $x_j\ge2n_0$ eventually. Discarding finitely many terms proves the lemma.
\end{proof}

\section{Relations between the two counting functions}
\label{sec:relations}

We first show that the gap in $A$ forces a gap in $B$. This will allow us to estimate both counting functions on $[y,y+x]$.

\begin{lemma}
\label{lem:gap-propagation}
Let $A,B\subseteq\Nzero$ satisfy \eqref{eq:eventual} for some integer $n_0\ge1$. If $x<y$ are consecutive elements of $A$ with
\[
  x\ge2n_0,\qquad y\ge3x+n_0,
\]
then
\begin{equation}
  B\cap[y-x+n_0,\ y+x+n_0-1]=\varnothing.
  \label{eq:B-gap}
\end{equation}
\end{lemma}

\begin{proof}
Put
\[
  f=\ind_{A\cap[0,x]},\qquad g=\ind_B,\qquad h=\ind_{\Nzero}.
\]
Define the reflected function $\widetilde f$ by $\widetilde f(n)=f(x-n)$, and put
\[
  D=h-\widetilde f*g.
\]
For $n\ge0$, the coefficient of the convolution is
\begin{equation}
  (\widetilde f*g)(n)
  =\#\{(a,b)\in(A\cap[0,x])\times B\mid b-a=n-x\}.
  \label{eq:difference-count}
\end{equation}
This follows by writing $a=x-j$ and $b=n-j$ in the convolution sum. Thus $D(n)$ is $1$ minus the number of these representations of $n-x$.

Since $A\cap(x,y)=\varnothing$, every representation of an integer below $y$ uses $a\le x$. Hence $f*g=h+\delta$ on $[0,y-1]$. Using $f*D=f*h-\widetilde f*(f*g)$, we obtain
\begin{equation}
  f*D=(f-\widetilde f)*h-\widetilde f*\delta
  \qquad\text{at every index }0\le n<y.
  \label{eq:convolution-identity}
\end{equation}
Only indices at most $n$ occur in the substitution, so the identity is valid throughout this range.

For $n\ge x$, both $(f*h)(n)$ and $(\widetilde f*h)(n)$ equal $A(x)$. Also, $\widetilde f$ vanishes outside $[0,x]$, so $(\widetilde f*\delta)(n)=0$ for $n\ge x+n_0$. It follows from \eqref{eq:convolution-identity} that
\begin{equation}
  (f*D)(n)=0\qquad(x+n_0\le n<y).
  \label{eq:convolution-zero}
\end{equation}

To determine where $D$ vanishes, first note that by Lemma~\ref{lem:differences}, a repeated difference can involve only elements less than $n_0$. In view of \eqref{eq:difference-count}, this gives
\begin{equation}
  D(n)<0\quad\Longrightarrow\quad
  x-n_0+1\le n\le x+n_0-1.
  \label{eq:negative-support}
\end{equation}
In particular, $D(n)\ge0$ whenever $n\ge x+n_0$.

Take an integer $u$ with $x+n_0\le u\le2x+n_0-1$. The gap hypothesis gives
\[
  u+x\le3x+n_0-1<y.
\]
Using \eqref{eq:convolution-zero} at the index $u+x$, we obtain
\[
  0=(f*D)(u+x)=\sum_{\substack{a\in A\\a\le x}}D(u+x-a).
\]
Every argument is at least $x+n_0$, so every summand is nonnegative by \eqref{eq:negative-support}. The term with $a=x$ is $D(u)$, and therefore
\begin{equation}
  D(u)=0\qquad(x+n_0\le u\le2x+n_0-1).
  \label{eq:right-zero}
\end{equation}

We claim that the larger interval also satisfies
\begin{equation}
  D(n)=0\qquad(n_0\le n\le2x+n_0-1).
  \label{eq:full-zero}
\end{equation}
Otherwise, choose the largest integer $d$ in this interval for which $D(d)\ne0$. Relation \eqref{eq:right-zero} gives $d\le x+n_0-1$. Consequently,
\[
  x+n_0\le x+d\le2x+n_0-1<y.
\]
Equation \eqref{eq:convolution-zero} at $x+d$ now gives
\[
  0=(f*D)(x+d)=\sum_{\substack{a\in A\\a\le x}}D(x+d-a).
\]
For $a<x$, the argument satisfies
\[
  d<x+d-a\le2x+n_0-1.
\]
All such terms vanish by the choice of $d$. The term for $a=x$ is $D(d)$, which is nonzero. This contradiction proves \eqref{eq:full-zero}.

We can now prove the gap in $B$. By \eqref{eq:difference-count} and \eqref{eq:full-zero}, every integer in
\[
  [n_0-x,x+n_0-1]
\]
has a representation as $b'-a$ with $a\le x$. If $b$ belonged to the interval in \eqref{eq:B-gap}, then $b-y$ would lie in this interval of differences. Hence
\[
  b-y=b'-a
\]
for some $a\in A\cap[0,x]$ and $b'\in B$. The pairs $(y,b)$ and $(a,b')$ are distinct, and $y\ge n_0$. This contradicts Lemma~\ref{lem:differences} and proves \eqref{eq:B-gap}.
\end{proof}

The next lemma uses this gap to control $B$ and then compares the two segments of $A$.

\begin{lemma}
\label{lem:relations}
Under the assumptions of Lemma~\ref{lem:gap-propagation}, let $C=E+n_0$ be as in \eqref{eq:fixed-errors}. For every integer $t$ with $0\le t\le x$,
\begin{equation}
  \abs{A(y+t)-A(x)-A(t)}\le C
  \label{eq:A-repetition}
\end{equation}
and
\begin{equation}
  \abs{A(x)B(y+t)-y}\le C.
  \label{eq:B-value}
\end{equation}
\end{lemma}

\begin{proof}
Put $m=A(x)$. Since $x\ge2n_0$ and $n_0\ge1$, the gap in \eqref{eq:B-gap} contains $[y,y+x]$. Thus
\begin{equation}
  B(y+t)=B(y-1)\qquad(0\le t\le x).
  \label{eq:B-constant}
\end{equation}

Consider the $mB(y-1)$ pairs in
\[
  (A\cap[0,x])\times(B\cap[0,y-1]).
\]
Every representation of an integer below $y$ lies in this rectangle, because $A\cap(x,y)=\varnothing$. Conversely, \eqref{eq:B-gap} implies that every $b\in B\cap[0,y-1]$ satisfies $b\le y-x+n_0-1$, so every sum from the rectangle is at most $y+n_0-1$. Counting representations with their multiplicities and using \eqref{eq:sum-count}, we obtain
\begin{equation}
  y+c=\sum_{n=0}^{y-1}r(n)
  \le mB(y-1)
  \le\sum_{n=0}^{y+n_0-1}r(n)=y+n_0+c.
  \label{eq:rectangle-bounds}
\end{equation}
Since $|c|\le E$, equations \eqref{eq:B-constant} and \eqref{eq:rectangle-bounds} imply \eqref{eq:B-value}.

It remains to compare the two segments of $A$. Put $f=\ind_{A\cap[0,x]}$, $g=\ind_B$, $h=\ind_{\Nzero}$ and
\[
  f_y(n)=\ind_A(y+n)\qquad(n\ge0),
\]
and extend $f_y$ by zero to negative integers. For $n_0\le n\le x+n_0-1$, consider the unique representation of $y+n$. If it used an element $a\le x$, then
\[
  y-x+n_0\le y+n-a\le y+x+n_0-1,
\]
so the required element of $B$ would belong to the empty interval \eqref{eq:B-gap}. It must therefore use an element $a\ge y$. Thus
\[
  (f_y*g)(n)=1\qquad(n_0\le n\le x+n_0-1).
\]
For $0\le n<n_0$, the convolution $(f_y*g)(n)$ is either $0$ or $1$, since it counts a subset of the representations of $y+n$, whose representation is unique. Thus there is a function $\delta_y$, supported on $[0,n_0-1]$ and taking values in $\{-1,0\}$, such that
\begin{equation}
  f_y*g=h+\delta_y
  \qquad\text{at indices }0\le n\le x+n_0-1.
  \label{eq:next-convolution}
\end{equation}
The old segment satisfies
\[
  f*g=h+\delta\qquad\text{at indices }0\le n<y.
\]
Substituting these identities into $f_y*(f*g)=f*(f_y*g)$ gives
\begin{equation}
  ((f_y-f)*h)(t)=(f*\delta_y)(t)-(f_y*\delta)(t)
  \qquad(0\le t\le x).
  \label{eq:segment-comparison}
\end{equation}
The substitution uses only indices at most $t\le x$. Since $A(y-1)=m$, the left side is
\[
  A(y+t)-m-A(t).
\]
On the right, $f$ and $f_y$ take values in $\{0,1\}$, and hence
\[
  |(f*\delta_y)(t)|\le n_0,\qquad
  |(f_y*\delta)(t)|\le E.
\]
Equation \eqref{eq:segment-comparison} therefore proves \eqref{eq:A-repetition} with the same fixed constant $C=E+n_0$.
\end{proof}

\section{Proof of Theorem \ref{thm:main}}
\label{sec:main-proof}

\begin{proof}
Suppose that \eqref{eq:main-hypothesis} holds but \eqref{eq:excess} fails. Choose $\kappa$ strictly between $1$ and $3/2$ so that \eqref{eq:global-bound} holds. Lemmas~\ref{lem:bounded} and \ref{lem:gaps} give an integer $n_0\ge1$ satisfying \eqref{eq:eventual} and, after interchanging the sets if necessary, consecutive elements $x_j<y_j$ satisfying \eqref{eq:large-gaps}.

Fix $C=E+n_0$ as in \eqref{eq:fixed-errors}, and put $m_j=A(x_j)$. Since $A$ is infinite and $x_j\to\infty$, we have $m_j\to\infty$. Apply Lemma~\ref{lem:relations} at $x=x_j$, $y=y_j$, and $t=x_j$. Equations \eqref{eq:A-repetition} and \eqref{eq:B-value} give
\[
  A(y_j+x_j)\ge2m_j-C,\qquad
  m_jB(y_j+x_j)\ge y_j-C.
\]
For large $j$, the two lower bounds are positive. Using \eqref{eq:global-bound} at $t=y_j+x_j$, we obtain
\begin{equation}
\begin{aligned}
  \kappa
  &\ge\frac{A(y_j+x_j)B(y_j+x_j)}{y_j+x_j}\\
  &\ge\left(2-\frac{C}{m_j}\right)
       \left(1-\frac{C}{y_j}\right)
       \frac{y_j}{y_j+x_j}\\
  &\ge\frac34
       \left(2-\frac{C}{m_j}\right)
       \left(1-\frac{C}{y_j}\right).
\end{aligned}
  \label{eq:final-bound}
\end{equation}
The last inequality follows from $y_j\ge3x_j+n_0$. As $m_j$ and $y_j$ both tend to infinity, letting $j\to\infty$ in \eqref{eq:final-bound} gives $\kappa\ge3/2$. This contradicts the choice of $\kappa$ and proves Theorem~\ref{thm:main}.
\end{proof}

\section{Sharpness}
\label{sec:sharpness}

We recall a digit example and verify the endpoint by direct counting. It is a special case of Ma's construction \cite[Lemma~2.1 and equation~(3.1)]{Ma22}; the existence of examples at $3/2$ was proved earlier by Chen and Fang \cite{CF11}.

Let
\begin{equation}
  A=\left\{\sum_{i\ge0}\varepsilon_i4^i\ \middle|\
       \varepsilon_i\in\{0,1\},\
       \text{only finitely many }\varepsilon_i\ne0\right\},
  \qquad B=2A.
  \label{eq:digit-sets}
\end{equation}
Every base-$4$ digit is uniquely the sum of a digit in $\{0,1\}$ and one in $\{0,2\}$, with no carries. Thus every nonnegative integer has exactly one representation in $A+B$. At $X_k=4^k-1$, counting the first $k$ digits gives
\begin{equation}
  A(X_k)=B(X_k)=2^k,\qquad A(X_k)B(X_k)-X_k=1.
  \label{eq:unit-excess}
\end{equation}

To bound the product ratio, write $v=\sum_i\varepsilon_i2^i$ and put $a_v=\sum_i\varepsilon_i4^i$. These $a_v$ list $A$ in increasing order. Since $\sum_{j<i}\varepsilon_j2^j\le2^i-1$,
\begin{equation}
\begin{aligned}
  v^2
  &=\sum_i\varepsilon_i4^i
    +2\sum_i\varepsilon_i2^i\sum_{j<i}\varepsilon_j2^j\\
  &\le\sum_i\varepsilon_i4^i+2\sum_i\varepsilon_i2^i(2^i-1)
  \le3a_v.
\end{aligned}
  \label{eq:digit-lower-bound}
\end{equation}

Fix $k\ge0$ and put $H=2^k$. In $[4^k,4^{k+1})$, the elements of $A$ are $4^k+a_v$, and those of $B$ are $2(4^k+a_v)$, with $0\le v<H$. As $a_v\le(4^k-1)/3$, the $A$ block ends before the $B$ block begins. Counting at these points gives
\[
\begin{aligned}
  A(4^k+a_v)&=H+v+1,& B(4^k+a_v)&=H,\\
  A(2(4^k+a_v))&=2H,& B(2(4^k+a_v))&=H+v+1.
\end{aligned}
\]
At either point, the product ratio is
\[
  \frac{H(H+v+1)}{H^2+a_v}
  \le\frac{H(H+v)}{H^2+v^2/3}+\frac1H
  \le\frac32+\frac1H,
\]
where the last inequality follows from
\[
  \frac32\left(H^2+\frac{v^2}{3}\right)-H(H+v)
  =\frac{(H-v)^2}{2}\ge0.
\]
Between consecutive elements of $A\cup B$, the numerator is constant and the denominator increases. The same bound therefore holds throughout $[4^k,4^{k+1})$, giving
\begin{equation}
  \limsup_{x\to\infty}\frac{A(x)B(x)}{x}\le\frac32.
  \label{eq:example-upper}
\end{equation}

For the reverse inequality, take $z_k=2(4^k-1)/3$, the largest element of $B$ using the first $k$ digits. Then $A(z_k)=B(z_k)=2^k$, so
\[
  \frac{A(z_k)B(z_k)}{z_k}
  =\frac{3\cdot4^k}{2(4^k-1)}\longrightarrow\frac32.
\]
Together with \eqref{eq:unit-excess}, this proves that the constant $3/2$ is best possible and that the strict inequality in Theorem~\ref{thm:main} is necessary.

\section*{Acknowledgments and Declaration of AI assistance}
When reading Yang and Zhao's paper, the author believed that the large gap forced another set to emerge with a corresponding gap (their paper only focus on one set), and used ChatGPT to explore the possible relationships. OpenAI Codex was used to assist with literature searches, English editing, the organization of the manuscript, and the checking and revision of some proof details, including the treatment of the finite initial exceptions. The author is responsible for the mathematical arguments, references, and final text.

The author thanks Professor Yang and Professor Zhao for their suggestions. The author expresses his gratitude to his graduate supervisor Huixi Li for his guidance.

\end{document}